\documentclass[11pt,draft]{article}
\usepackage[utf8]{inputenc}
\usepackage{amsmath}
\usepackage{enumerate}
\usepackage{graphicx}
\usepackage{graphics}
\usepackage[english,spanish]{babel}
\usepackage{amsmath,amssymb,amsfonts,amstext,amsthm}
\usepackage{mathrsfs}
\usepackage{latexsym}
\usepackage{multicol}

\newtheorem{defin}{Definición}[section]
\newtheorem{prop}{Proposición}[section]
\newtheorem{teor}{Teorema}[section]
\newtheorem{corol}{Corolario}[section]

\newtheorem{lema}{Lema}[section]
\theoremstyle{definition}

\newtheorem{remark}{Observación}[section]

\title{\textbf{Un teorema de Serre para el anillo de series formales torcidas}}
\author{Edward O. Latorre Acero \footnote{eolatorrea@unal.edu.co}\\ Seminario de Álgebra Constructiva - SAC$^2$\\
Departamento de Matemáticas\\ Universidad
Nacional de Colombia, Sede Bogotá}
\date{}
\begin{document}
\maketitle
\begin{abstract}
  En este artículo se muestra que si $R$ es un anillo noetheriano y regular a izquierda tal que todo $R$-módulo izquierdo proyectivo finitamente generado es establemente libre, entonces lo mismo se cumple para la completación $R[[x;\sigma,\delta]]$ de cualquier extensión de Ore $R[x;\sigma,\delta]$ del anillo $R$. Se asumen adicionalmente ciertas condiciones naturales para $\sigma$ y $\delta$ que garantizan que tal completación esta bien definida. La demostración requiere comparar los respectivos grupos de Grothendieck $K_{0}$ de $R$ y de $S$ haciendo uso de algunos resultados previos encontrados en \cite{lam}, \cite{Mc}, \cite{SchV} y \cite{Ven}.\\
\begin{center}
 \textbf{Abstract}
\end{center}
  In this paper we prove that if $R$ is a left Noetherian and left regular ring such that all finitely generated projective left $R$-modules are stably free, then the same is true for the completion $R[[x;\sigma,\delta]]$ of any Ore extension $R[x;\sigma,\delta]$ of $R$. Additionally assumes certain natural conditions for $\sigma$  and $\delta$ for ensure that this completion are good defined. The proof compares the Grothendieck groups $K_{0}$ of $R$ and $S$ using some results find in \cite{lam}, \cite{Mc}, \cite{SchV} and \cite{Ven}.\\
\bigskip

\noindent \textit{Key words and phrases.} Noetherian and regular noncommutative rings, graded and filtered
rings and modules, skew power series ring, Serre's theorem, Grothendieck group.

\bigskip

\noindent 2000 \textit{Mathematics Subject Classification.}
Primary: 16S80, 16W35, 16S36, 16U20. Secondary: 16W50, 16E65.
\end{abstract}
\section{Introducción}

En este artículo se estudia el teorema de Serre sobre módulos establemente libres para el anillo $R[[x;\sigma,\delta]]$ de series formales torcidas, motivados por el correspondiente resultado considerado en \cite{lz} para una gama muy amplia de anillos y álgebras que incluyen estensiones de Ore, álgebras de operadores diferenciales, álgebras envolventes de álgebras de Lie, álgebras cuánticas, álgebras de difusión, entre muchos otros.\\

En la segunda sección se presenta el problema de Serre y el anillo de polinomios torcidos $T=R[x,\sigma,\delta]$ junto a su completación; el anillo de series formales torcidas $S=R[[x;\sigma,\delta]]$, el cual ha sido introducido en el estudio de la teoría de álgebras cuánticas \cite{Kash} y en la teoría de Iwasawa no conmutativa \cite{SchV} y \cite{Ven}. Haciendo uso de la técnica de filtración-graduación se presentan algunas propiedades homológicas que se conservan entre el anillo base y su extensión. Para la técnica de filtración-graduación se puede consultar \cite{Mc}.\\

En la tercera sección se hace una presentación del grupo de Grothendieck $K_{0}$ para un anillo $R$ en general, siguiendo la exposición dada por \cite{lam}. Se muestran dos resultados bien conocidos relativos a esta estructura, necesarios en la prueba central del presente artículo.\\

En la cuarta sección se muestra la versión del teorema de Serre para la completación $S$ de la extensión de Ore $T$. Exponemos algunos ejemplos de \cite{let} y \cite{LW} que permitan consolidar las ideas expuestas a lo largo del documento.

\section{El problema de Serre y las series formales torcidas}

En los 50's, J.P. Serre plantea como problema comprobar que cada fibrado vectorial sobre $A_{n}{k}$ (espacio afín del álgebra polinomial $A=k[t_{1},...,t_{n}]$, con $k$ cuerpo) es un fibrado trivial, o equivalentemente, que si al considerar un módulo proyectivo finitamente generado $P$ sobre el anillo $A=k[t_{1},...,t_{n}]$, entonces $P$ es libre. Diversos casos particulares fueron presentados reafirmando y condicionando este interrogante, para una descripción bien detallada e historicamente argumentada consultar \cite{lam}.\\

 El problema fue resuelto independientemente por D. Quillen y A. Suslin en enero de $1976$, para todo $n$, y para todo cuerpo $k$. Nuevas versiones han sido demostradas al plantear el problema de Serre en dos secciones:
\begin{enumerate}
  \item[i)] mostrar que todo $A$-módulo proyectivo es \textit{establemente libre}, lo cual se tiene probando que el grupo de Grothendieck $K_{0}(A)$ de módulos proyectivos es isomorfo a $\mathbb{Z}$.
  \item[ii)] mostrar que cualquier $A$-módulo proyectivo establemente libre es libre.
\end{enumerate}

El ánimo del presente artículo es ver bajo qué condiciones se tiene la parte $(i)$ del problema de Serre si consideramos un anillo $R$ (no necesariamente conmutativo) y $S=R[[x;\sigma,\delta]]$ la completación de su extensión de Ore $T=R[x;\sigma,\delta]$; es decir, ver cómo se hereda la propiedad $PSF$, (véase la definición \ref{p1}) del anillo base $R$ al anillo $S$ cuando se usa la técnica de graduación.\\

\subsection{El anillo de series formales torcidas}

Sea $R$ un anillo con identidad (no necesariamente conmutativo) y sean $\sigma$ un endomorfismo de $R$ y $\delta$ una $\sigma$-derivación (e.d., una aplicación que satisface $\delta(a+b)=\delta(a)+\delta(b)$ y $\delta(ab)=\sigma(a)\delta(b)+\delta(a)b$, para todo $a,b\in R$). Consideremos la extensión de Ore $T:=R[x;\sigma,\delta]$ del anillo $R$ (véase [\cite{Mc}]. Recordemos que $T$ es un anillo de tipo polinomial con producto dado por $xr=\sigma(r)x+ \delta(r)$, $\forall r\in R$.\\

El anillo de series de potencias formales torcidas
\begin{center}
  $R[[x;\sigma,\delta]]$
\end{center}
está definido como el anillo cuyo conjunto subyacente coincide con
el conjunto de series de potencias formales $R[[x]]$, pero donde la multiplicación verifica la relación
$xr=\sigma(r)x+ \delta(r)$, $\forall r\in R$, satisfaciendo la siguiente regla:
\begin{align*}
  \sum\limits_{i\geq 0}x^{i}a_{i}=\sum\limits_{j\geq 0}(\sum\limits_{i\geq
  j}M_{i-j,j}(\delta,\sigma)(a_{i}))x^{j}.
\end{align*}
 La multiplicación de dos elementos en $S$ está explícitmante dada por:
\begin{align*}
  (\sum\limits_{j\geq 0}a_{j}x^{j})(\sum\limits_{l\geq
  0}b_{l}x^{l})=\sum\limits_{m\geq 0}(\sum\limits_{n=0}^{m}\sum\limits_{j\geq
  n}a_{j}M_{j-n,n}(\delta,\sigma)(b_{m-n}))x^{m}.
\end{align*}
considerando $M_{k,l}(Y,Z)$ (véase \cite{SchV}) para cualquier entero $k,l\geq 0$, como la suma de todos los monomios no conmutativos en las dos variables $Y,Z$ con $k$ factores $Y$ y $l$ factores $Z$.\\

De lo anterior, se define a el anillo $S$ como el $R$-módulo izquierdo de todas
las series de potencias formales $\sum\limits_{i\geq 0}a_{i}x^{i}$
sobre $R$ en la variable $x$. Como el anillo $T$ es un $R$-submódulo de $S$, para extender su estructura de anillo por continuidad al anillo $S$, es necesario asegurar que las sumas que conforman los
coeficientes determinados por la fórmula anterior convergen a un
elemento del anillo $R$, razón por lo cual en la literatura se consideran los siguientes casos:
\begin{enumerate}
\item[i)] Al tomar $\delta=0$ entonces $x^{n}r=\sigma^{n}(r)x^{n}$ e.d. la
suma $\sum\limits_{i=0}^{n}a_{i}\sigma^{i}(b_{n-i})$ es finita en
$R$, por lo tanto el producto está bien definido y el anillo existe.

\item[ii)] Al tomar $R$ como anillo completo con respecto a la topología $I$-ádica, para $I$ algún ideal bilátero $\sigma$-estable $(\sigma(I)\subseteq I)$, y además tener que:
\begin{center}
  $\delta(R)\subseteq I,\,\,\delta(I)\subseteq I^{2}$.
\end{center}
Se obtiene por inducción sobre $k$ que $\delta(I^{k})\subseteq
I^{k+1} \,\forall k\geq 0$ ($I^{0}=R$ por definición); la condición
sobre $\delta$ se exige para poder controlar los términos libres
$r+\delta(r)+\delta^{2}(r)+\cdots$ que no necesariamente están
definidos en $R$.
\end{enumerate}

La condición a utilizar para la estructura $S$ en éste documento es la $ii)$, pues permite tener una visión global y completa de todos los términos. Considerando $\mathfrak{A_{k}}$ como el conjunto de todos los monomios no conmutativos en las variables $Y,Z$ con exactamente $k$ factores $Y$ y al conjunto $\mathfrak{A_{\geq l}}:=\bigcup\limits_{k\geq l} \mathfrak{A_{k}}$ motivamos la siguiente definición:

\begin{defin}
  Una $\sigma$-derivación $\delta$ es \textbf{$\sigma$-nilpotente} si para todo $n\geq 1$, existe
  $m\geq 1$ tal que $M(\delta,\sigma)(R)\subseteq I^{n}$ para cualquier $M\in \mathfrak{A_{\geq m}}$.
\end{defin}

\begin{remark}
Si $\delta$ satisface que $\delta(R)\subseteq I$ y $\delta(I)\subseteq I^{2}$ entonces $\delta(I^{n})\subseteq I^{n+1}$ es cierto para $n\geq 0$ y como $\sigma(I)=I$ se sigue que $\delta$ es $\sigma$-nilpotente. Condición suficiente para la existencia del anillo de series formales torcidas.
\end{remark}

En \cite{Ven} se muestra que $S$ es isomorfo como $R$-módulo al producto directo infinito de copias del anillo $R$ por potencias de la variable $x$ e.d. $S:=\prod\limits_{i=0}^{\infty}Rx^{i}$. Para el resto de ésta sección se asume que $R$ es un anillo local completo con respecto a la topología generada por su radical de Jacobson $J(R)$. Así, la filtración inducida por $J(R)$ sobre $R$ induce la filtración $G_{k}$ de $S$ determinada por:

\begin{align}
  G_{k}=\prod\limits_{i\geq 0}J(R)^{k-i}x^{i}, \forall k\geq 0.
\end{align}

\begin{lema}\label{q1}
Para $S=R[[x;\sigma,\delta]]$ se tiene que:
  \begin{enumerate}
    \item[i)] Cada $G_{k}$ es un ideal bilátero de $S$.
    \item[ii)] $\{G_{k}\}_{k\geq 0}$ es una filtración de $S$.
    \item[iii)] Si $\cap_{k\geq 0} J(R)^{k}=0$ entonces $S=\underleftarrow{\lim}S/G_{k}$.
    \item[iv)] El anillo graduado del anillo filtrado $S$ es
    \begin{center}
      $Gr_{G_{k}}(S)\cong (Gr_{J(R)}(R))[\overline{x};\overline{\sigma}]$, donde $\overline{x}$ denota el símbolo principal de $x$ en $Gr_{G_{k}}(S)$.
    \end{center}
  \end{enumerate}
\end{lema}

\begin{proof}
(i) Para ver que $G_{k}$ es ideal bilátero de $S$ basta verificar que el producto de un elemento de $S$ por un elemento de $G_{k}$ esta en $G_{k}$; sin embargo basta ver que sucede con productos de la variable $x$, entonces sea $f=\sum f_{i}x^{i}\in G_{k}$. Por la regla de
multiplicación de escalares en $S$ se tiene que $xf$ se descompone en la suma
\begin{center}
  $xf=\sum\limits_{i=0}^{\infty}\delta(f_{i})x^{i}+(\sum\limits_{i=1}^{\infty}\sigma(f_{i-1})x^{i-1})x$,
\end{center}
donde las constantes del primer sumando estan en $J(R)^{k+1}$ y las del segundo en $J(R)^{k}$ finalmente contenidos en $G_{k}$.\\
  
(ii) Nótese que $G_{k}S=\prod\limits_{i=0}^{\infty}J(R)^{k-i}x^{i}=(\prod\limits_{i=k}^{\infty}Rx^{i})\times J(R)x^{k-1}\times\cdots\times J(R)^{k}x^{0}$ por lo tanto el hecho de ser filtración se debe a que $\{J(R)^{k}\}_{k}$ es filtración para $R$ y por la regla de multiplicación de $S$.\\

(iii) Si $\cap_{k\geq 0}(J(R))^{k}=0$ entonces también $\cap_{k\geq
0}G_{k}=0$, y como $\ker(\phi)=\cap_{k\geq
0}G_{k}$ el homomorfismo de anillos
$\phi:S\longrightarrow \widehat{S}$ es un isomorfismo, de donde $S$ es un anillo completo respecto a la filtración $G_{k}$, así $S=\underleftarrow{\lim}S/G_{k}$.\\
(iv) Observemos que:
\begin{center}
  $G_{k}/G_{k+1}\cong \prod\limits_{i=0}^{k}(J(R)^{i}/J(R)^{i+1})\overline{x}^{k-i}$
\end{center}
se sigue que
\begin{align*}
  \bigoplus\limits_{k=0}^{\infty}G_{k}/G_{k+1}&\cong \bigoplus\limits_{i,l\geq
  0}J(R)^{i}/J(R)^{i+1}\overline{x}^{l}\\
  &\cong \bigoplus\limits_{i\geq 0}Gr_{J(R)}R \overline{x}^{i}\\
  &\cong (Gr_{J(R)}R)[\overline{x};\overline{\sigma}].
  \end{align*}
\end{proof}

\begin{lema}\label{q2}
Suponiendo que $\delta$ es $\sigma$-nilpotente con
$\delta(R)\subseteq J(R)$ y que $Gr_{J(R)}R$ es noetheriano, entonces $S$ es noetheriano.
\end{lema}

\begin{proof}
Es bien sabido que el anillo de polinomios torcidos $Gr_{G_{k}}S\cong Gr_{J(R)}R[\overline{x};\overline{\sigma}]$, sobre el anillo noetheriano $Gr_{J(R)}R$, es noetheriano, (véase \cite{Mc}). Como el anillo $S$ es completo con respecto a la topología $G_{k}$-ádica y es un hecho general que un anillo filtrado completo es noetheriano si su  anillo graduado asociado lo es, se sigue la propiedad.
\end{proof}

\subsection{Anillos regulares}

\begin{defin}
  Un anillo $R$ se dice \textrm{regular} a izquierda si cada $M\in \mathfrak{M}(R)$ (módulos f.g. sobre $R$) admite una resolución $0\rightarrow P_{n}\rightarrow \cdots \rightarrow P_{0}\rightarrow M\rightarrow 0$, donde $P_{i}\in \mathfrak{B}(R)$ (módulos proyectivos f.g. sobre $R$), y $n$ depende de $M$.
\end{defin}

Nótese que la definición anterior es equivalente a decir que cada $M\in \mathfrak{M}(R)$ tiene dimensión proyectiva finita (\cite{Mc}). El siguiente lema es clásico en estas estructuras y constituye un ejemplo del uso de la técnica de graduación como herramienta fundamental para demostrar algunas propiedades homológicas.

\begin{lema}\label{q3}
  Sea $R$ un anillo filtrado.
  \begin{enumerate}
    \item[i)] Si $Gr(R)$ es noetheriano a izquierda entonces $R$ es noetheriano a izquierda.
    \item[ii)] Si $Gr(R)$ es regular a izquierda, entonces $R$ es regular a izquierda.
  \end{enumerate}
\end{lema}

\section{El grupo de Grothendieck $K_{0}$}

Partiendo de la siguiente definición y haciendo uso de la K-teoría algebraica de Quillen se calculará el \textit{grupo de Grothendieck} $K_{0}(S)$ verificando que, bajo ciertas condiciones del anillo $R$, se tiene $K_{0}(S)\cong K_{0}(R)$.

\begin{defin}\label{p1}
  Un anillo $R$ cumple la propiedad $PSF$ si cada módulo proyectivo f.g. es \textbf{establemente libre}. Un $R$-módulo izquierdo $M$ es establemente libre si existen enteros $r,s \geq 0$ tales que $R^{r}\cong R^{s}\oplus M$.
\end{defin}

Para $P\in \mathfrak{B}(R)$, escribimos $(P)$ para la clase de isomorfismo de $P$, e.d. todos los $R$-módulos isomorfos a $P$. El \textit{grupo de Grothendieck} $K_{0}(R)$ es un grupo aditivo abeliano generado por las clases $(P)$ junto a cierta relación natural. Para precisarla, sean:

\begin{enumerate}
  \item[i)]$G:=$grupo abeliano libre generado por $(P)$: $P\in \mathfrak{B}(R)$,
  \item[ii)]$H:=$subgrupo de $G$ generado por $(P\oplus Q)-(P)-(Q)$: $P,Q\in \mathfrak{B}(R)$,
  \item[iii)] $K_{0}(R)=G/H$,
  \item[iv)] $[P]=$imagen de $(P)$ en $K_{0}(R)$.
\end{enumerate}

Entonces tenemos que $[P\oplus Q]=[P]+[Q]\in K_{0}(R)$ siempre que $P,Q\in \mathfrak{B}(R)$. De manera más general, si existe una sucesión exacta
\begin{align*}
  0\rightarrow P_{n} \rightarrow P_{n-1}\rightarrow \cdots \rightarrow P_{0} \rightarrow 0,
\end{align*}
donde $P_{i}\in \mathfrak{B}(R)$, entonces $\sum(-1)^{i}[P_{i}]=0\in K_{0}(R)$. En general un elemento de $K_{0}(R)$ tiene la forma:
\begin{align*}
  z&=[P_{1}]+\cdots + [P_{m}]- [Q]- \cdots - [Q_{n}]\\
  &= [P]-[Q],
\end{align*}
donde $P=P_{1}\oplus\cdots\oplus P_{m}$, $Q=Q_{1}\oplus\cdots\oplus Q_{n}$.

\begin{prop}
  Para $P,Q\in \mathfrak{B}(R)$, las siguientes afirmaciones son equivalentes:
  \begin{enumerate}
    \item[i)] $[P]=[Q]\in K_{0}(R)$;
    \item[ii)] Existe $T\in \mathfrak{B}(R)$ tal que $P\oplus T\cong Q\oplus T$ (e.d,$P,Q$ son establemente isomorfos);
    \item[iii)] Existe un número natural $t$ tal que $P\oplus R^{t}\cong Q\oplus R^{t}$.

  \end{enumerate}
\end{prop}

\begin{proof}
  Es claro que $iii)\Leftrightarrow ii)\Rightarrow i)$, entonces se necesita asumir $i)$, y probar $ii)$. Por $i)$, se tiene que $(P)-(Q)\in H$, entonces
  \begin{align*}
    (P)-(Q)=\sum_{i} {(P_{i}\oplus Q_{i})-(P_{i})-(Q_{i})}-\sum_{j}{(P'_{j}\oplus Q'_{j})-(P'_{j})-(Q'_{j})},
  \end{align*}
  donde todos los módulos proyectivos f.g. estan en $\mathfrak{B}(R)$. Luego,
  \begin{align*}
    (P)+\sum_{i}{(P_{i})+(Q_{i})}+\sum_{j}(P'_{j}\oplus Q'_{j})= (Q)+\sum_{i}(P_{i}\oplus Q_{i})+ \sum_{j}{(P'_{j})+(Q'_{j})}.
  \end{align*}
  Ahora, como $G$ es libre, $\sum(M_{\alpha})=\sum (N_{\beta})$ implica que $\bigoplus M_{\alpha}\cong \bigoplus N_{\beta}$. Por lo tanto, se obtiene que $P\oplus T\cong Q\oplus T$ con $T=\bigoplus_{i}{P_{i}\oplus Q_{i}}\oplus\bigoplus_{j}{P'_{j}\oplus Q'_{j}}$.
\end{proof}
\begin{corol}\label{q5}
   $P\in\mathfrak{B}(R)$ es establemente libre si, y solo si, $[P]\in\mathbb{Z}\cdot[R]$. (e.d., $K_{0}(R)= \mathbb{Z}\cdot[R]$ si, y solo si, todo $R$-módulo proyectivo f.g. es establemente libre.)
\end{corol}
\begin{proof}
  Al considerar $P\oplus R^{m}\cong R^{n}$ con $m < n$, entonces $[P]=[R^{n}]-[R^{m}]=(n-m)\cdot[R]$. De manera análoga, si $[P]=r\cdot[R],r\in \mathbb{Z}$. Tomando $s\in \mathbb{Z}$ tal que $r+s\geq 0$, se tiene que $[P\oplus R^{s}]=(r+s)\cdot [R]=[R^{r+s}]\in K_{0}(R)$ y por la proposición 3.1., $P\oplus R^{s}\oplus R^{t}\cong R^{r+s+t}$ para algún $t$.
 \end{proof}
  
A continuación presentamos una herramienta muy poderosa debida a Quillen que sirve para calcular el grupo de Grothendieck de anillos filtrados.  

\begin{teor}{\textbf{Quillen's theorem.\\}}
  Sea $S$ un anillo filtrado, $B=Gr(S)$, y $R=S_{0}=B_{0}$. Supongamos que $B$ es un anillo noetheriano y regular a izquierda, además que $B$ es un $R$-módulo plano derecho. Entonces la inclusion $R\hookrightarrow S$ induce el isomorfismo $K_{0}(R)\cong K_{0}(S)$.
\end{teor}

\section{El teorema de Serre}

En esta última sección presentamos el resultado central del artículo. La prueba se basa en lo expuesto en las secciones anteriores y en el siguiente lema.

\begin{lema}\label{q4}
  Sea $R$ regular y noetheriano a izquierda, $\sigma$ automorfismo y $\delta$ una derivación $\sigma$-nilpotente con $\delta(R)\subseteq J(R)$ , $S=R[[x;\sigma,\delta]]$ filtrado con la filtración (1), $B=Gr(S)$ su anillo graduado asociado. Entonces $K_{0}(R)\cong K_{0}(S)$.
\end{lema}

\begin{proof}
  Por el lema \ref{q1} se sabe que $B=Gr(S)\cong(Gr_{J(R)}(R))[\overline{x};\overline{\sigma}]$. Utilizando los lemas \ref{q2} y \ref{q3} se tiene que $S$ es regular y noetheriano a izquierda. Además, $B_{0}=S_{0}=R$; $B$ es $R$-plano ya que es $R$-libre, entonces podemos aplicar el Lema de Quillen.
\end{proof}

\begin{teor}{\textbf{Serre's theorem.\\}}
  Sea $R$ regular y noetheriano a izquierda y $PSF$, $\sigma$ automorfismo y $\delta$ una derivación $\sigma$-nilpotente con $\delta(R)\subseteq J(R)$, entonces $S$ es $PSF$.
\end{teor}

\begin{proof}
  Es inmediato por el lema \ref{q4} que bajo las hipótesis $K_{0}(R)\cong K_{0}(S)$. Por tanto, si el anillo $R$ es $PSF$ entonces $K_{0}(R)=\mathbb{Z}\cdot [R]$ por el corolario \ref{q5}, y por el isomorfismo de anillos se infiere que $K_{0}(S)=\mathbb{Z}\cdot [R]$. Así, $S$ es $PSF$ si el anillo base $R$ también lo es.
\end{proof}

\end{document}